\documentclass[leqno,12pt]{amsart} 
\usepackage
[top=30truemm,bottom=30truemm,left=25truemm,right=25truemm]
{geometry}
\usepackage{amssymb}
\usepackage{amsmath}
\usepackage{amsthm}
\usepackage{amscd}
\usepackage{mathrsfs}
\usepackage{graphicx}
\usepackage{url}
  \makeatletter

\@addtoreset{equation}{section}
\@namedef{subjclassname@2020}{\textup{2020} Mathematics Subject Classification}
\makeatother
\theoremstyle{plain} 
\newtheorem{theorem}{\indent\bf Theorem}[section]

\newtheorem{proposition}[theorem]{\indent\bf Proposition}

\newtheorem{conjecture}[theorem]{\indent\bf Conjecture}
\theoremstyle{definition} 
\newtheorem{definition}[theorem]{\indent\bf Definition}

\newcommand{\dbar}{\overline{\partial}}

\newcommand{\R}{\mathbb{R}}

\newcommand{\C}{\mathbb{C}}

\newcommand{\B}{\mathbb{B}}

\newcommand{\Rea}{\mathrm{Re}}

\begin{document}
\pagestyle{plain}
\thispagestyle{plain}

\title[Note on characterizing pluriharmonic functions]
{A note on characterizing pluriharmonic functions via the Ohsawa--Takegoshi extension theorem}

\author[T. INAYAMA]{Takahiro INAYAMA}
\address{Department of Mathematics\\
Faculty of Science and Technology\\
Tokyo University of Science\\
2641 Yamazaki, Noda\\
Chiba, 278-8510\\
Japan
}
\email{inayama\_takahiro@rs.tus.ac.jp}
\email{inayama570@gmail.com}
\subjclass[2020]{32A10, 32A36}
\keywords{ 
Ohsawa--Takegoshi extension theorem, $L^2$-extension, plurisubharmonic function, pluriharmonic function, $L^2$-extension index.
}

\begin{abstract}
 For a continuous function, we prove that the function is pluriharmonic if and only if the equality part of the optimal Ohsawa--Takegoshi $L^2$-extension theorem is satisfied with respect to the metric having the function as a weight.  
 This partially resolves the conjecture proposed by the author. 
\end{abstract}


\maketitle
\setcounter{tocdepth}{2}

\section{Introduction}\label{sec:intro}

On a one-dimensional complex domain, as is well known, a subharmonic function is characterized by the mean value inequality, and when the equality of the inequality holds, the function becomes a harmonic function. Subharmonic functions play an important role in complex analysis and geometry, and are also used in important theorems such as H\"ormander's $L^2$-estimate \cite{Hor65} and Ohsawa--Takegoshi's $L^2$-extension theorem \cite{OT87}.

On the other hand, recent research has revealed that the fact that the optimal Ohsawa--Takegoshi $L^2$-extension theorem holds itself guarantees the subharmonicity of the weight. This property is called the minimal extension property \cite{HPS18} or the optimal $L^p$-extension property \cite{DNW21,DNWZ23} in a general setting, and has been widely studied and applied by various experts (cf. \cite{GZ15,HPS18,DNW21,DNWZ23,Ina22a,Ina22b}). In other words, subharmonic functions can be characterized by the inequality part of the optimal Ohsawa--Takegoshi $L^2$-extension theorem. 
Based on the analogy with the above, we propose the following conjecture in \cite[Appendix A]{Ina22b}: 

\begin{conjecture}\label{conj:pluriharmonic}
Let $\varphi$ be an upper semi-continuous function on a domain $\Omega\subset \C^n$. 
Then the following are equivalent: 
\begin{enumerate}
    \item $\varphi$ is pluriharmonic.
    \item $\varphi>-\infty$ and for any holomorphic cylinder $a+P_{r,s,A}$, where $(a,r,s,A)\in \Omega_{\widetilde{\delta}}$ {\rm (}see the notation below{\rm )}, there exists a unique holomorphic function $f$ on $a+P_{r,s,A}$ satisfying $f(a)=1$ and 
    \[
    \int_{a+P_{r,s,A}} |f|^2e^{-\varphi} \leq |P_{r,s,A}|e^{-\varphi(a)},
    \]
    where $|P_{r,s,A}|$ is the volume of $P_{r,s,A}$. 
\end{enumerate}
\end{conjecture}


We proved this conjecture in the case that $\varphi$ was smooth by using the theory of the $L^2$-extension index in \cite{Ina22b}. 
Our aim of this paper is to give an affirmative answer to the above conjecture for a \textit{continuous} function $\varphi$. 
To clarify the claim of the theorem, several symbols will be prepared. 
Let $\Delta_r=\{ z\in \C \mid |z|<r\}$, $\B^m_s=\{ z\in\C^m\mid |z|<s \}$ and $P_{r,s,A}=A(\Delta_r\times \B^{n-1}_s)$ for $r,s>0$ and $A\in \mathbf{U}(n)$, where $\mathbf{U}(n)$ is the set of all $n$-dimensional unitary groups.
Here $P_{r,s,A}$ is a holomorphic cylinder.
We let $\Omega_{\widetilde{\delta}}=\{ (a,r,s,A)\in \Omega\times\R_{>0}\times\R_{>0}\times \mathbf{U}(n)\mid a+P_{r,s,A}\subset \Omega \}$. 
Then the main theorem can be stated as follows. 

\begin{theorem}\label{mainthm:pluriharmonic}
Let $\varphi$ be a \textbf{continuous} function on $\Omega$. Then $\varphi$ is pluriharmonic if and only if for any $(a,r,s,A)\in \Omega_{\widetilde{\delta}}$, there exists a \textit{unique} holomorphic function $f$ on $a+P_{r,s,A}$ satisfying $f(a)=1$ and 
\[
\int_{a+P_{r,s,A}} |f|^2e^{-\varphi}\leq |P_{r,s,A}| e^{-\varphi(a)}.
\]
\end{theorem}

Note that the above $f$ satisfies the equality 
\[
\int_{a+P_{r,s,A}} |f|^2e^{-\varphi}= |P_{r,s,A}| e^{-\varphi(a)}.
\]
For the proof, we use the terms of the $L^2$-extension index and the characterization of log-plurisubharmonicity.



\vskip3mm
{\bf Acknowledgment. } 
The author is supported by Japan Society for the Promotion of Science, Grant-in-Aid for Research Activity Start-up (Grant No. 21K20336) and Grant-in-Aid for Early-Career Scientists (Grant No. 23K12978).
He also expresses his gratitude to Dr. Wang Xu for fruitful discussions, sharing various ideas and his generous response.

\section{The proof of main theorem}

In order to prove the main result, we prepare several notions. 
First, we review the minimal extension property \cite{HPS18} or the optimal $L^2$-extension property \cite{DNW21,DNWZ23}.
In this paper, we follow the formulation of Deng, Ning, Wang and Zhou.

\begin{definition}[{\cite[Definition 1.1]{DNW21}}]\label{def:optimall2}
Let $\Omega$ be a domain in $\C^n$ and $\varphi$ be an upper semi-continuous function. 
Then we say that $\varphi$ satisfies {\it the optimal $L^2$-extension property} if for any $(a,r,s,A)\in \Omega_{\widetilde{\delta}}$, there exists a holomorphic function $f$ on $a+P_{r,s,A}$ such that $f(a)=1$ and 
\[
\int_{a+P_{r,s,A}} |f|^2e^{-\varphi}\leq |P_{r,s,A}|e^{-\varphi(a)}.
\]
\end{definition}

The following theorem is important in relation to this definition.

\begin{theorem}[{\cite[Theorem 1.6]{DNW21}}]\label{thm:optimall2}
    Keep the setting. 
    Then $\varphi$ is plurisubharmonic if and only if $\varphi$ satisfies the optimal $L^2$-extension property. 
\end{theorem}

Next, we explain the $L^2$-extension index introduced by the author in \cite{Ina22b}. 
Here we adopt a slightly extended definition.

\begin{definition}\label{def:L2indexgeneral}
	Let $\varphi$ be a function $\varphi\colon\Omega\to [-\infty, \infty]$.
	Then we define the \textit{$L^2$-extension index} $L_\varphi$ of $\varphi$ on $\Omega_{\widetilde{\delta}}$ as follows: for $(a,r,s,A)\in \Omega_{\widetilde{\delta}}$, if $-\infty <\varphi(a)<+\infty$,
	\begin{align*}
	L_{\varphi}(a,r,s,A):&= \frac{1}{|P_{r,s,A}|K_{P_{r,s,A}, \varphi}(a)}\\
	&=\inf \left\{ \frac{\int_{a+P_{r,s,A}}|f|^2e^{-\varphi}}{|P_{r,s,A}|e^{-\varphi(a)}} ~\middle|~ f\in A^2(a+P_{r,s,A}, \varphi) \And f(a)=1 \right\},
	\end{align*}
	if $\varphi(a)=+\infty$, $L_\varphi(a,r,s,A)=+\infty$ and if $\varphi(a)=-\infty$, $L_{\varphi}(a,r,s,A)=0$.
	Here, $K_{{P_{r,s,A}},\varphi}$ is the weighted Bergman kernel on $P_{r,s,A}$ with respect to $\varphi$, $A^2(a+P_{r,s,A}, \varphi)=\{ f\in\mathcal{O}(a+P_{r,s,A})\mid \int_{a+P_{r,s,A}} |f|^2e^{-\varphi}<+\infty \}$ and $\mathcal{O}(a+P_{r,s,A})$ is the space of all holomorphic functions on $a+P_{r,s,A}$.
\end{definition}

By using this notion, we can rephrase the optimal $L^2$-extension property as follows: if $L_\varphi \leq 1$ for an upper semi-continuous function, then $\varphi$ is plurisubharmonic. 
Our goal is to prove the next proposition that is paired with the above result. 

\begin{proposition}\label{prop:lsc}
    Let $\varphi$ be a lower semi-continuous function on a domain $\Omega\subset \C^n$ with $\varphi \not\equiv +\infty$. 
    If $L_\varphi\geq 1$, then $\varphi$ is plurisuperharmonic.
\end{proposition}



\begin{proof}
    We use the proposition that for a non-negative function $v$ with $v\not \equiv 0$, $\log v$ is plurisubharmonic if and only if $v e^{2\Rea g}$ is plurisubharmonic for every polynomial $g$. 
    We take any polynomial $g$ and any $(a,r,s,A)\in \Omega_{\widetilde{\delta}}$. We may assume that $\varphi(a)<+\infty$. 
    If $L_\varphi \geq 1$, it holds that
    \[
    \int_{a+P_{r,s,A}}|e^g|^2e^{-\varphi}\geq |P_{r,s,A}| |e^{g(a)}|^2e^{-\varphi(a)}.
    \]
    Since $-\varphi$ is upper semi-continuous, we can say that $e^{-\varphi}e^{2\Rea g}$ is plurisubharmonic (see Lemma 3.1 in \cite{DNW21}). Hence, $-\varphi$ is plurisubharmonic. 
\end{proof}

By using this proposition, we can prove Theorem \ref{mainthm:pluriharmonic}. 

\begin{proof}[Proof of Theorem \ref{mainthm:pluriharmonic}]
If $\varphi$ is pluriharmonic, take a holomorphic function $h$ on $a+P_{r,s,A}$ satisfying $2\Rea (h)=\varphi$ and use the argument in \cite[Section 5]{Ina22b}. 
Then we only prove the if part. 
Note that the assumption in Theorem \ref{mainthm:pluriharmonic} says that $L_\varphi\equiv 1$. 
	Since $L_\varphi\leq 1$ and $\varphi$ is upper semi-continuous, Theorem \ref{thm:optimall2} implies $\varphi$ is plurisubharmonic. 
	Also, since $L_\varphi\geq 1$ and $\varphi$ is lower semi-continuous, Proposition \ref{prop:lsc} implies that $\varphi$ is plurisuperharmonic.
\end{proof}

I have discussed Conjecture \ref{conj:pluriharmonic} with Wang Xu. 
On May 15, I sent him the above proof. 
Then, by using the linearity of certain minimal $L^2$-integrals, Xu showed that the assumption that $\varphi$ is lower semi-continuous is not needed and sent me the proof on May 16. 
About a month later, with Zhuo Liu, Xu consequently obtained the result that the upper semi-continuity of $\varphi$ is also unnecessary and sent the manuscript \cite{LZ23} to me on June 26. 
Although their result is literally a generalization of my theorem, they encouraged me to write this paper on my result as a first important step.
I would like to thank them for their consideration and warm encouragement.



\end{document}